\documentclass[reqno]{amsart}

\usepackage{amssymb}
\usepackage{hyperref}
\usepackage{mathtools}
\usepackage[all]{xy}
\usepackage{ifthen}
\usepackage{xargs}
\usepackage{bussproofs}
\usepackage{turnstile}
\usepackage{etex}

\hypersetup{colorlinks=true,linkcolor=blue}

\renewcommand{\turnstile}[6][s]
    {\ifthenelse{\equal{#1}{d}}
        {\sbox{\first}{$\displaystyle{#4}$}
        \sbox{\second}{$\displaystyle{#5}$}}{}
    \ifthenelse{\equal{#1}{t}}
        {\sbox{\first}{$\textstyle{#4}$}
        \sbox{\second}{$\textstyle{#5}$}}{}
    \ifthenelse{\equal{#1}{s}}
        {\sbox{\first}{$\scriptstyle{#4}$}
        \sbox{\second}{$\scriptstyle{#5}$}}{}
    \ifthenelse{\equal{#1}{ss}}
        {\sbox{\first}{$\scriptscriptstyle{#4}$}
        \sbox{\second}{$\scriptscriptstyle{#5}$}}{}
    \setlength{\dashthickness}{0.111ex}
    \setlength{\ddashthickness}{0.35ex}
    \setlength{\leasturnstilewidth}{2em}
    \setlength{\extrawidth}{0.2em}
    \ifthenelse{%
      \equal{#3}{n}}{\setlength{\tinyverdistance}{0ex}}{}
    \ifthenelse{%
      \equal{#3}{s}}{\setlength{\tinyverdistance}{0.5\dashthickness}}{}
    \ifthenelse{%
      \equal{#3}{d}}{\setlength{\tinyverdistance}{0.5\ddashthickness}
        \addtolength{\tinyverdistance}{\dashthickness}}{}
    \ifthenelse{%
      \equal{#3}{t}}{\setlength{\tinyverdistance}{1.5\dashthickness}
        \addtolength{\tinyverdistance}{\ddashthickness}}{}
        \setlength{\verdistance}{0.4ex}
        \settoheight{\lengthvar}{\usebox{\first}}
        \setlength{\raisedown}{-\lengthvar}
        \addtolength{\raisedown}{-\tinyverdistance}
        \addtolength{\raisedown}{-\verdistance}
        \settodepth{\raiseup}{\usebox{\second}}
        \addtolength{\raiseup}{\tinyverdistance}
        \addtolength{\raiseup}{\verdistance}
        \setlength{\lift}{0.8ex}
        \settowidth{\firstwidth}{\usebox{\first}}
        \settowidth{\secondwidth}{\usebox{\second}}
        \ifthenelse{\lengthtest{\firstwidth = 0ex}
            \and
            \lengthtest{\secondwidth = 0ex}}
                {\setlength{\turnstilewidth}{\leasturnstilewidth}}
                {\setlength{\turnstilewidth}{2\extrawidth}
        \ifthenelse{\lengthtest{\firstwidth < \secondwidth}}
            {\addtolength{\turnstilewidth}{\secondwidth}}
            {\addtolength{\turnstilewidth}{\firstwidth}}}
        \ifthenelse{\lengthtest{\turnstilewidth < \leasturnstilewidth}}{\setlength{\turnstilewidth}{\leasturnstilewidth}}{}
    \setlength{\turnstileheight}{1.5ex}
    \sbox{\turnstilebox}
    {\raisebox{\lift}{\ensuremath{
        \makever{#2}{\dashthickness}{\turnstileheight}{\ddashthickness}
        \makehor{#3}{\dashthickness}{\turnstilewidth}{\ddashthickness}
        \hspace{-\turnstilewidth}
        \raisebox{\raisedown}
        {\makebox[\turnstilewidth]{\usebox{\first}}}
            \hspace{-\turnstilewidth}
            \raisebox{\raiseup}
            {\makebox[\turnstilewidth]{\usebox{\second}}}
        \makever{#6}{\dashthickness}{\turnstileheight}{\ddashthickness}}}}
        \mathrel{\usebox{\turnstilebox}}}

\newcommand{\axlabel}[1]{(#1) \phantomsection \label{ax:#1}}
\newcommand{\axtag}[1]{\label{ax:#1} \tag{#1}}
\newcommand{\axref}[1]{(\hyperref[ax:#1]{#1})}

\newcommand{\newref}[4][]{
\ifthenelse{\equal{#1}{}}{\newtheorem{h#2}[hthm]{#4}}{\newtheorem{h#2}{#4}[#1]}
\expandafter\newcommand\csname r#2\endcsname[1]{#3~\ref{#2:##1}}
\expandafter\newcommand\csname R#2\endcsname[1]{#4~\ref{#2:##1}}
\expandafter\newcommand\csname n#2\endcsname[1]{\ref{#2:##1}}
\newenvironmentx{#2}[2][1=,2=]{
\ifthenelse{\equal{##2}{}}{\begin{h#2}}{\begin{h#2}[##2]}
\ifthenelse{\equal{##1}{}}{}{\label{#2:##1}}
}{\end{h#2}}
}

\newref[section]{thm}{Theorem}{Theorem}
\newref{lem}{Lemma}{Lemma}
\newref{prop}{Proposition}{Proposition}
\newref{cor}{Corollary}{Corollary}

\theoremstyle{definition}
\newref{defn}{Definition}{Definition}
\newref{example}{Example}{Example}

\theoremstyle{remark}
\newref{remark}{Remark}{Remark}

\newcommand{\deq}{\equiv}

\newcommand{\cat}[1]{\mathbf{#1}}

\newcommand{\ccat}{\cat{CCat}}
\newcommand{\algtt}{\cat{TT}}
\newcommand{\substTh}{\mathbb{S}}
\newcommand{\Mod}[1]{#1\text{-}\cat{Mod}}

\newcommand{\Th}{\cat{Th}}
\newcommand{\St}{\cat{St}}
\newcommand{\PSt}{\cat{PSt}}
\newcommand{\cSt}[1][c]{#1\text{-}\St}
\newcommand{\emptyCtx}{*}

\numberwithin{figure}{section}

\newcommand{\pb}[1][dr]{\save*!/#1-1.2pc/#1:(-1,1)@^{|-}\restore}

\begin{document}

\title{Algebraic Presentations of Dependent Type Theories}

\author{Valery Isaev}

\begin{abstract}
In this paper, we propose an abstract definition of dependent type theories as essentially algebraic theories.
One of the main advantages of this definition is its composability: simple theories can be combined into more complex ones,
and different properties of the resulting theory may be deduced from properties of the basic ones.
We define a category of algebraic dependent type theories which allows us not only to combine theories but also to consider equivalences between them.
We also study models of such theories and show that one can think of them as contextual categories with additional structure.
\end{abstract}

\maketitle

 \makeatletter
    \providecommand\@dotsep{5}
  \makeatother

\section{Introduction}

Type theories with dependent types originally were defined by Per Martin-L\"{o}f, who introduced several versions of the system \cite{MLTT72,MLTT73,MLTT79}.
There were also several theories and extensions of Martin-L\"{o}f's theory proposed by different authors (\cite{CoC,luo94} to name a few).
These theories may have different inference rules, different computation rules, and different constructions.
Many of these theories have common parts and similar properties,
but the problem is that there is no general definition of a type theory such that all of these theories would be a special case of this definition,
so that their properties could be studied in general and applied to specific theory when necessary.
In this paper we propose such a definition based on the notion of essentially algebraic theories.

Another problem of the usual way of defining type theories is that they are not composable.
Some constructions in type theories are independent of each other (such as $\Pi$, $\Sigma$, and $Id$ types),
and others may dependent on other constructions (such as universes),
so we could hope that we can study these constructions independently (at least if they are of the first kind)
and deduce properties of combined theory from the properties of these basic constructions.
But this is not the way it is usually done.
For example, constructing models of dependent type theories is a difficult task because of the so called coherence problem.
There are several proposed solutions to this problems, but the question we are interested in is how to combine them.
Often only the categorical side of the question is considered,
but some authors do consider specific theories \cite{streicher,pitts},
and the problem in this case is that their work cannot be applied to other similar theories (at least formally).

When defining a type theory there are certain questions to be addressed regarding syntactic traits of the theory.
One such question is how many arguments to different construction can be omitted and how to restore them when constructing a model of the theory.
For example, we want to define application as a function of two arguments $app(f,a)$, but sometimes it is convenient to have additional arguments which allows to infer a type of $f$.
It is possible to prove that additional information in the application term may be omitted (for example, see \cite{streicher}), but it is a nontrivial task.
Another question of this sort is whether we should use a typed or an untyped equality.
Typed equality is easier to handle when defining a model of the theory, but untyped is closer to actual implementation of the language.
Algebraic approach allows us to separate these syntactic details from essential aspects of the theory.

Yet another problem is that some constructions may be defined in several different ways.
For example, $\Sigma$ types can be defined using projections (\rexample{sigma-eta}) and using an eliminator (\rexample{sigma-no-eta}).
The question then is whether these definitions are equivalent in some sense.
The difficulty of this question stems from the fact that some equivalences may hold in one definition judgmentally, but in the other only propositionally;
so it may be difficult (or impossible) to construct a map from the first version of the definition to the second one.

In this paper, using the formalism of essentially algebraic theories, we introduce the notion of
\emph{algebraic dependent type theories} which provide a possible solution the problems described above.
We define a category of algebraic dependent type theories.
Coproducts and more generally colimits in this category allow us to combine simple theories into more complex ones.
For example, the theory with $\Sigma$, $\Pi$ and $Id$ types may be described as coproduct $T_\Sigma \amalg T_\Pi \amalg T_{Id}$
where $T_\Sigma$, $T_\Pi$ and $T_{Id}$ are theories of $\Sigma$, $\Pi$ and $Id$ types respectively.

There is a natural notion of a model of an essentially algebraic theory.
Thus the algebraic approach to defining type theories automatically equips every type theory with a (locally presentable) category of its models.
We will show that models of the initial theory are precisely contextual categories,
and that models of an arbitrary theory are contextual categories with an additional structure (which depends on the theory).
An example of a general construction that works for all theories with enough structure is the construction of a model structure on the category of models described in \cite{alg-models}.

Since we have a category of type theories, there is a natural notion of equivalence between them, namely the isomorphism.
In most cases this equivalence is too strong, so it is necessary to consider weaker notions of equivalence, but in some cases it might be useful.
For example, if two theories differ only by the amount of arguments to some of the constructions,
then they are isomorphic (assuming omitted arguments can be inferred from the rest).
A weaker notion of equivalence of theories is Morita equivalence.
Two theories are Morita equivalent if there is a Quillen equivalence between the categories of models of these theories.
We will not consider this notion in this paper.

Usually, we can use all constructions of a type theory in every context.
We consider an additional structure on theories which allows us to do this.
We call theories with this additional structure \emph{prestable}.
Then, an algebraic dependent type theory is a prestable theory with substitutions which commute with every operation in the theory.
We also consider \emph{stable} theories in which all axioms are stable under context extensions.
If we think of models of a prestable theory as some sort of category with some additional structure,
then the prestable structure allows us to pass to slices of this category.
Then a prestable theory is stable if not only the category itself but also every slice category has this additional structure.

The paper is organized as follows.
In section 2, we define the category of partial Horn theories and discuss its properties.
In section 3, we define an example of partial Horn theory and prove that the category of its models is equivalent to the category of contextual categories.
In section 4, we define algebraic type theories and describe a simplified version of the syntax that can be used with these theories.
In section 5, we give a few standard examples of such theories. In particular, we show that the construction that adds a universe to the system is functorial.

\section{Partial Horn theories}
\label{sec:PHT}

There are several equivalent ways of defining essentially algebraic theories (\cite{LPC}, \cite{GAT}, \cite{PHL}, \cite[D 1.3.4]{elephant}).
We use approach introduced in \cite{PHL} under the name of partial Horn theories since it is the most convenient one.
There is a structure of a category on partial Horn theories.
A \emph{generalized morphism} between theories $\mathbb{T}$ and $\mathbb{T}'$ is a model of $\mathbb{T}$ in $\mathcal{C}_{\mathbb{T}'}$,
where $\mathcal{C}_{\mathbb{T}'}$ is the classifying category for $\mathbb{T}'$.
We will work with theories that have some fixed set of sorts.
Thus we need a notion of morphisms which preserve sorts.
Of course, we could restrict the notion of a generalized morphism, but there is another definition of morphisms, which is more explicit.

Let us recall the basic definitions from \cite{PHL}.
A many sorted first-order signature $(\mathcal{S},\mathcal{F},\mathcal{P})$ consists of a set $\mathcal{S}$ of sorts,
a set $\mathcal{F}$ of function symbols and a set $\mathcal{P}$ of predicate symbols.
Each function symbol $\sigma$ is equipped with a signature of the form $\sigma : s_1 \times \ldots \times s_k \to s$, where $s_1$, \ldots $s_k$, $s$ are sorts.
Each predicate symbol $R$ is equipped with a signature of the form $R : s_1 \times \ldots \times s_k$.

An atomic formula is an expression either of the form $t_1 = t_2$ or of the form $R(t_1, \ldots t_n)$,
where $R$ is a predicate symbol and $t_1$, \ldots $t_n$ are terms.
We abbriviate $t = t$ to $t\!\downarrow$.
A Horn formula is an expression of the form $\varphi_1 \land \ldots \land \varphi_n$, where $\varphi_1$, \ldots $\varphi_n$ are atomic formulas.
A sequent is an expression of the form $\varphi \sststile{}{x_1, \ldots x_n} \psi$, where $x_1$, \ldots $x_n$ are variables
and $\varphi$ and $\psi$ are Horn formulas such that $FV(\varphi) \cup FV(\psi) \subseteq \{ x_1, \ldots x_n \}$.
A \emph{partial Horn theory} consists of a signature and a set of Horn sequents in this signature.

Let $V$ be an $\mathcal{S}$-set.
Then the $\mathcal{S}$-set of terms of $\mathbb{T}$ with free variables in $V$ will be denoted by $Term_\mathbb{T}(V)$.
The set of fomulas of $\mathbb{T}$ with free variables in $V$ will be denoted by $Form_\mathbb{T}(V)$.

An $\mathcal{S}$-set $M$ is a collection of sets $\{ M_s \}_{s \in \mathcal{S}}$.
An interpretation $M$ of a signature $(\mathcal{S},\mathcal{F},\mathcal{P})$ is an $\mathcal{S}$-set $M$
together with a collection of \emph{partial} functions $M(\sigma) : M_{s_1} \times \ldots \times M_{s_k} \to M_s$
for every function symbol $\sigma : s_1 \times \ldots \times s_k \to s$ of $\mathbb{T}$
and relations $M(R) \subseteq M_{s_1} \times \ldots \times M_{s_k}$ for every predicate symbol $R : s_1 \times \ldots \times s_k$.
A model of a partial Horn theory $\mathbb{T}$ is an interpretation of the underlying signature such that the axioms of $\mathbb{T}$ hold in this interpretation.
The category of models of $\mathbb{T}$ will be denoted by $\Mod{\mathbb{T}}$.

The rules of \emph{partial Horn logic} are listed below.
A \emph{theorem} of a partial Horn theory $\mathbb{T}$ is a sequent derivable from $\mathbb{T}$ in this logic.
\begin{center}
$\varphi \sststile{}{V} \varphi$ \axlabel{b1}
\qquad
\AxiomC{$\varphi \sststile{}{V} \psi$}
\AxiomC{$\psi \sststile{}{V} \chi$}
\RightLabel{\axlabel{b2}}
\BinaryInfC{$\varphi \sststile{}{V} \chi$}
\DisplayProof
\qquad
$\varphi \sststile{}{V} \top$ \axlabel{b3}
\end{center}

\medskip
\begin{center}
$\varphi \land \psi \sststile{}{V} \varphi$ \axlabel{b4}
\qquad
$\varphi \land \psi \sststile{}{V} \psi$ \axlabel{b5}
\qquad
\AxiomC{$\varphi \sststile{}{V} \psi$}
\AxiomC{$\varphi \sststile{}{V} \chi$}
\RightLabel{\axlabel{b6}}
\BinaryInfC{$\varphi \sststile{}{V} \psi \land \chi$}
\DisplayProof
\end{center}

\medskip
\begin{center}
$\sststile{}{x} x\!\downarrow$ \axlabel{a1}
\qquad
$x = y \land \varphi \sststile{}{V,x,y} \varphi[y/x]$ \axlabel{a2}
\end{center}

\medskip
\begin{center}
\AxiomC{$\varphi \sststile{}{V} \psi$}
\RightLabel{, $x \in FV(\varphi)$ \axlabel{a3}}
\UnaryInfC{$\varphi[t/x] \sststile{}{V,V'} \psi[t/x]$}
\DisplayProof
\end{center}
\medskip

Note that this set of rules is equivalent to the one described in \cite{PHL}.
In particular, the following sequents are derivable if $x \in FV(t)$:
\begin{align*}
R(t_1, \ldots t_k) & \sststile{}{V} t_i = t_i \axtag{a4} \\
t_1 = t_2 & \sststile{}{V} t_i = t_i \axtag{a4'} \\
t[t'/x]\!\downarrow & \sststile{}{V} t' = t' \axtag{a5}
\end{align*}

We will use the following abbreviations:
\begin{align*}
\varphi \sststile{}{V} t \cong s & \Longleftrightarrow \varphi \land t\!\downarrow\,\sststile{}{V} t = s \text{ and } \varphi \land s\!\downarrow\,\sststile{}{V} t = s \\
\varphi \ssststile{}{V} \psi & \Longleftrightarrow \varphi \sststile{}{V} \psi \text{ and } \psi \sststile{}{V} \varphi
\end{align*}

Let $\mathbb{T}$ be a partial Horn theory.
A \emph{restricted term} of $\mathbb{T}$ is a term $t$ together with a formula $\varphi$.
We denote such a restricted term by $t|_\varphi$.
The $\mathcal{S}$-set of restricted terms with free variables in $V$ will be denoted by $RTerm_\mathbb{T}(V)$.
If we think of terms as representations for partial functions,
then we can think of a restricted term $t|_\varphi$ as a restriction of the partial function represented by $t$ to a subset of its domain.
We will use the following abbreviations:
\begin{align*}
R(t_1|_{\varphi_1}, \ldots t_k|_{\varphi_k}) & \Longleftrightarrow R(t_1, \ldots t_k) \land \varphi_1 \land \ldots \land \varphi_k \\
t|_\varphi = s|_\psi & \Longleftrightarrow t = s \land \varphi \land \psi \\
t|_\varphi\!\downarrow & \Longleftrightarrow t\!\downarrow\!\land \varphi \\
\chi \sststile{}{V} t|_\varphi \cong s|_\psi & \Longleftrightarrow \chi \land t|_\varphi\!\downarrow\,\sststile{}{V} t = s \land \psi \text{ and } \chi \land s|_\psi\!\downarrow\,\sststile{}{V} t = s \land \varphi
\end{align*}
We will say that formulas $\varphi$ and $\psi$ are equivalent if the following sequents are derivable:
\[ \varphi \ssststile{}{FV(\varphi) \cup FV(\psi)} \psi \]
We will say that restricted terms $t$ and $t'$ are equivalent if the following sequents are derivable:
\[ \sststile{}{FV(t) \cup FV(t')} t \cong t' \]

Let $\mathbb{T}$ and $\mathbb{T}'$ be partial Horn theories with the same set of sorts $\mathcal{S}$.
An \emph{interpretation} of $\mathbb{T}$ in $\mathbb{T}'$ is a function $f$ such that the following conditions hold:
\begin{enumerate}
\item For every function symbol $\sigma : s_1 \times \ldots \times s_k \to s$ of $\mathbb{T}$,
the function $f$ defines a restricted term $f(\sigma)$ of $\mathbb{T}'$ of sort $s$ such that $FV(f(\sigma)) = \{ x_1 : s_1, \ldots x_k : s_k \}$.
\item For every predicate symbol $P : s_1 \times \ldots \times s_k$,
the function $f$ defines a formula $f(P)$ of $\mathbb{T}'$ such that $FV(f(P)) = \{ x_1 : s_1, \ldots x_k : s_k \}$.
\item For every axiom $\varphi \sststile{}{V} \psi$ of $\mathbb{T}$, the sequent $f(\varphi) \sststile{}{V} f(\psi)$ is derivable in $\mathbb{T}'$.
\end{enumerate}

We will say that interpretations $f$ and $f'$ are equivalent if,
for every predicate symbol $P : s_1 \times \ldots \times s_k$ of $\mathbb{T}$, the formulas $f(P)$ and $f'(P)$ are equivalent
and, for every function symbol $\sigma : s_1 \times \ldots \times s_k \to s$ of $\mathbb{T}$, the terms $f(\sigma)$ and $f'(\sigma)$ are also equivalent.
A \emph{morphism} of theories $\mathbb{T}$ and $\mathbb{T}'$ is an equivalence class of interpretations.

The identity morphisms are defined in the obvious way.
To define the composition of morphisms, we need to extend the definition of a function $f : \mathbb{T} \to \mathbb{T}'$ to terms and formulas.
Let $t$ be a term of $\mathbb{T}$ of sort $s$.
Then we define a restricted term $f(t)$ of $\mathbb{T}'$ by induction on $t$.
If $t = x$ is a variable, then let $f(t) = x$.
If $t = \sigma(t_1, \ldots t_k)$, $f(\sigma) = t'|_\varphi$ and $f(t_i) = t'_i|_{\varphi_i}$,
then let $f(t) = t'[t'_1/x_1, \ldots t'_k/x_k]|_{\varphi[t'_1/x_1, \ldots t'_k/x_k] \land \varphi_1 \land \ldots \land \varphi_k}$.

Let $\varphi$ be a formula of $\mathbb{T}$.
Then we define a formula $f(\varphi)$ of $\mathbb{T}'$.
If $\varphi$ equals to $t_1 = t_2$ and $f(t_i)$ equals to $t'_i|_{\varphi_i}$, then we define $f(\varphi)$ as $t'_1 = t'_2 \land \varphi_1 \land \varphi_2$.
If $\varphi = R(t_1, \ldots t_k)$, $f(R) = \varphi'$ and $f(t_i) = t'_i|_{\varphi_i}$,
then we define $f(\varphi)$ as $\varphi'[t'_1/x_1, \ldots t'_k/x_k] \land \varphi_1 \land \ldots \land \varphi_k$.
For every restricted term $t|_\varphi$ of $\mathbb{T}$, we define $f(t|_\varphi)$ as $f(t)|_{f(\varphi)}$.

Now, we can define the composition of $f : \mathbb{T} \to \mathbb{T}'$ and $g : \mathbb{T}' \to \mathbb{T}''$
as follows: $(g \circ f)(S) = g(f(S))$ for every symbol $S$ of $\mathbb{T}$.
It is easy to see that this definition respect the equivalence of morphisms.

It is obvious that, for every morphism $f : \mathbb{T} \to \mathbb{T}'$ of theories, we have $f \circ id_\mathbb{T} = id_{\mathbb{T}'} \circ f = f$.
Note that for every morphisms $f : \mathbb{T} \to \mathbb{T}'$ and $g : \mathbb{T}' \to \mathbb{T}''$ and every term $t$,
the restricted terms $g(f(t))$ and $(g \circ f)(t)$ are equivalent.
This is easy to do by induction on $t$.
Similarly, for every formula $\varphi$ of $\mathbb{T}$, the formulas $g(f(\varphi))$ and $(g \circ f)(\varphi)$ are equivalent.
It follows that the composition is associative.
The category of partial Horn theory with $\mathcal{S}$ as the set of sorts will be denoted by $\Th_\mathcal{S}$.
Its objects are tuples $(\mathcal{F},\mathcal{P},\mathcal{A})$, where $\mathcal{F}$ is a set of function symbols,
$\mathcal{P}$ is a set of predicate symbols, and $\mathcal{A}$ is a set of axioms.

\begin{prop}[th-cocomplete]
The category $\Th_\mathcal{S}$ is cocomplete.
\end{prop}
\begin{proof}
First, let $\{ \mathbb{T}_i \}_{i \in S} = \{ (\mathcal{F}_i,\mathcal{P}_i,\mathcal{A}_i) \}_{i \in S}$ be a set of theories.
Then we can define its coproduct $\coprod\limits_{i \in S} \mathbb{T}_i$ as the theory
$(\coprod\limits_{i \in S} \mathcal{F}_i, \coprod\limits_{i \in S} \mathcal{P}_i, \coprod\limits_{i \in S} \mathcal{A}_i)$.
Morphisms $f_i : \mathbb{T}_i \to \coprod\limits_{i \in S} \mathbb{T}_i$ are defined in the obvious way.
It is easy to see that the universal property of coproducts holds.

Now, let $f,g : \mathbb{T}_1 \to \mathbb{T}_2$ be a pair of morphisms of theories.
Then we can define their coequalizer $\mathbb{T}$ as the theory with the same set of function and predicate symbols as $\mathbb{T}_2$ and the set of axioms which consists of the axioms of $\mathbb{T}_2$
together with $\sststile{}{x_1, \ldots x_n} f(\sigma(x_1, \ldots x_n)) \cong g(\sigma(x_1, \ldots x_n))$ for each function symbols $\sigma$ of $\mathbb{T}_1$
and $f(R(x_1, \ldots x_n)) \ssststile{}{x_1, \ldots x_n} g(R(x_1, \ldots x_n))$ for each predicate symbols $R$ of $\mathbb{T}_1$.
Then we can define $e : \mathbb{T}_2 \to \mathbb{T}$ as the identity function on terms and formulas.
By construction, we have $e \circ f = e \circ g$.
If $h : \mathbb{T}_2 \to X$ is such that $h \circ f = h \circ g$, then it extends to a morphism $\mathbb{T} \to X$ since additional axioms are preserved by the assumption on $h$.
This extension is unique since $e$ is an epimorphism.
\end{proof}

\begin{prop}[func-mod]
For every morphism of theories $f : \mathbb{T} \to \mathbb{T}'$, there is a faithful functor $f^* : \Mod{T'} \to \Mod{T}$
such that $id_\mathbb{T}^*$ is the identity functor and $(g \circ f)^* = f^* \circ g^*$.
\end{prop}
\begin{proof}
If $M$ is a model of $\mathbb{T}'$, then $f^*(M)$ equals to $M$ as an $\mathcal{S}$-set.
For every symbol $S$ of $\mathbb{T}'$, we define $f^*(M)(S)$ as $M(f(S))$.
Then every morphism of models $M$ and $N$ of $\mathbb{T}'$ is also a morphism of $f^*(M)$ and $f^*(N)$.
These definitions determine a faithful functor $f^* : \Mod{T'} \to \Mod{T}$.
It is easy to see that these functors satisfy the required conditions.
\end{proof}

\section{Theory of substitutions}
\label{sec:T1}

In this section we define an example of partial Horn theories $\substTh$, which we call the theory of substitutions.
We also prove that the category of models of this theory is equivalent to the category of contextual categories
We will use this theory later to define algebraic dependent type theories.

\subsection{Definition of $\substTh$}
\label{sec:T1-def}

Let $\mathcal{C} = \{ ctx, tm \} \times \mathbb{N}$ be the set of sorts.
We will write $(ty,n)$ for $(ctx,n+1)$.
Sort $(tm,n)$ represents terms in contexts of length $n$, sort $(ctx,n)$ represents contexts of length $n$, and sort $(ty,n)$ represents types in contexts of length $n$.

There are two ways to define substitution: either to substitute the whole context (full substitution) or only a part of it (partial substitution).
Using ordinary type theoretic syntax the full substitution can be described by the following inference rule:
\begin{center}
\AxiomC{$A_1, \ldots A_n \vdash A\ type$}
\AxiomC{$\Gamma \vdash a_1 : A_1[]$ \quad \ldots \quad $\Gamma \vdash a_n : A_n[a_1, \ldots a_{n-1}]$}
\BinaryInfC{$\Gamma \vdash A[a_1, \ldots a_n]\ type$}
\DisplayProof
\end{center}
\medskip
The partial substitution is described by the following inference rule:
\begin{center}
\AxiomC{$\Gamma, A_1, \ldots A_n \vdash A\ type$}
\AxiomC{$\Gamma \vdash a_1 : A_1$ \quad \ldots \quad $\Gamma \vdash a_n : A_n[a_1, \ldots a_{n-1}]$}
\BinaryInfC{$\Gamma \vdash A[a_1, \ldots a_n]\ type$}
\DisplayProof
\end{center}
\medskip
The partial substitution was used in \cite{b-systems}, but we will use the full version since it is stronger.
To make these operations equivalent, we need to add another operation to the partial substitution, and even more axioms.
Thus our approach seems to be somewhat more convenient.

The set of function symbols of $\substTh$ consists of the following symbols:
\begin{align*}
\emptyCtx     & : (ctx,0) \\
ft_n          & : (ty,n) \to (ctx,n) \\
ty_n          & : (tm,n) \to (ty,n) \\
v_{n,i}       & : (ctx,n) \to (tm,n) \text{, } 0 \leq i < n \\
subst_{p,n,k} & : (ctx,n) \times (p,k) \times (tm,n)^k \to (p,n) \text{, } p \in \{ tm, ty \}
\end{align*}

Let $ft^i_n : (ctx,n+i) \to (ctx,n)$ and $ctx_{p,n} : (p,n) \to (ctx,n)$ be the following derived operations:
\begin{align*}
ft^0_n(A) & = A \\
ft^{i+1}_n(A) & = ft^i_n(ft_{n+i}(A)) \\
ctx_{ty,n}(t) & = ft_n(t) \\
ctx_{tm,n}(t) & = ft_n(ty_n(t))
\end{align*}

Auxiliary predicates $Hom_{n,k} : (ctx,n) \times (ctx,k) \times (tm,n)^k$ are defined as follows: $Hom_{n,k}(B, A, a_1, \ldots a_k)$ holds if and only if
\[ ty_n(a_i) = subst_{ty,n,i-1}(B, ft^{k-i}_i(A), a_1, \ldots a_{i-1}) \text{ for each } 1 \leq i \leq k \]
The idea is that a tuple of terms should represent a morphism in a contextual category.
So $Hom_{n,k}(B, A, a_1, \ldots a_k)$ holds if and only if $(a_1, \ldots a_k)$ is a morphism with domain $A$ and codomain $B$.
Note that if $Hom_{n,k}(B, A, a_1, \ldots a_k)$, then $ft_n(ty_n(a_i)) = B$.

The set of axioms of $\substTh$ consists of the axioms asserting that $(ctx,0)$ is trivial and the axioms we list below.
The following axioms describe when functions are defined:
\begin{align}
\label{ax:def-var}
                                             & \sststile{}{A}           v_{n,i}(A) \downarrow \\
\label{ax:def-subst}
Hom_{n,k}(B, ctx_{p,k}(a), a_1, \ldots a_k)  & \ssststile{}{B, a, a_i}  subst_{p,n,k}(B, a, a_1, \ldots a_k) \downarrow
\end{align}

The following axioms describe the ``typing'' of the constructions we have:
\begin{align}
\label{ax:type-var}
& \sststile{}{A} ty_n(v_{n,i}(A)) = subst_{ty,n,n-i-1}(A, ft^i_{n-i}(A), v_{n,n-1}(A), \ldots v_{n,i+1}(A)) \\
\label{ax:type-subst-ty}
& Hom_{n,k}(B, ft_k(A), a_1, \ldots a_k) \sststile{}{B, A, a_i} ft_n(subst_{ty,n,k}(B, A, a_1, \ldots a_k)) = B \\
\label{ax:type-subst-tm}
& \sststile{}{B, a, a_i} ty_n(subst_{tm,n,k}(B, a, a_1, \ldots a_k)) \cong subst_{ty,n,k}(B, ty_k(a), a_1, \ldots a_k)
\end{align}

The following axioms prescribe how $subst_{p,n,k}$ must be defined on indices ($v_{n,i}$):
\begin{align}
\label{ax:subst-var}
& \sststile{}{a}         subst_{p,n,n}(ctx_{p,n}(a), a, v_{n,n-1}(ctx_{p,n}(a)), \ldots v_{n,0}(ctx_{p,n}(a))) = a \\
\label{ax:var-subst}
& Hom_{n,k}(B, A, a_1, \ldots a_k) \sststile{}{B, a_i, A} subst_{tm,n,k}(B, v_{k,i}(A), a_1, \ldots a_k) = a_{k-i}
\end{align}

The last axiom say that substitution must be ``associative'':
\begin{align}
\label{ax:subst-subst}
& Hom_{n,k}(C, B, b_1, \ldots b_k) \land Hom_{k,m}(B, ctx_{p,m}(a), a_1, \ldots a_m) \sststile{}{C, b_i, B, a_i, a} \\ \notag
& subst_{p,n,k}(C, subst_{p,k,m}(B, a, a_1, \ldots a_m), b_1, \ldots b_k) = \\ \notag
& subst_{p,n,m}(C, a, subst_{tm,n,k}(C, a_1, b_1, \ldots b_k), \ldots subst_{tm,n,k}(C, a_m, b_1, \ldots b_k))
\end{align}

\subsection{Models of $\substTh$}

Here we show that the category of models of $\substTh$ is equivalent to the category of contextual categories.
First, we construct a functor $F : \Mod{\substTh} \to \ccat$.
Let $M$ be a model of $\substTh$.
Then the set of objects of level $n$ of $F(M)$ is $M_{(ctx,n)}$.
For each $A \in M_{(ctx,n)}$, $B \in M_{(ctx,k)}$ morphisms from $A$ to $B$ are tuples $(a_1, \ldots a_k)$ such that $a_i \in M_{(tm,n)}$ and $Hom_{n,k}(A, B, a_1, \ldots a_k)$.

For each $0 \leq i \leq n$ axiom~\eqref{ax:type-var} implies
\[ \sststile{}{A} Hom_{n,n-i}(A, ft^i_{n-i}(A), v_{n,n-1}(A), \ldots v_{n,i}(A)). \]
For each $A \in M_{(ctx,n)}$ we define $id_A : A \to A$ as tuple
\[ (v_{n,n-1}(A), \ldots v_{n,0}(A)) \]
and $p_A : A \to ft(A)$ as tuple
\[ (v_{n,n-1}(A), \ldots v_{n,1}(A)). \]

Now, we introduce some notation.
If $B \in M_{(ctx,n)}$, $a \in M_{(p,k)}$, and $f = (a_1, \ldots a_k) : B \to ctx_{p,k}(a)$ is a morphism, then we define $a[f] \in M_{(p,n)}$ as $subst_{p,n,k}(B, a, a_1, \ldots a_k)$.
By axiom \eqref{ax:def-subst} this construction is total.

If $A \in M_{(ctx,n)}$, $B \in M_{(ctx,k)}$, $C \in M_{(ctx,m)}$, $f : A \to B$, and $(c_1, \ldots c_m) : B \to C$,
    then we define composition $(c_1, \ldots c_m) \circ f$ as $(c_1[f], \ldots c_m[f])$.
The following sequence of equations shows that $(c_1, \ldots c_m) \circ f : A \to C$.
\begin{align*}
ty_n(c_i[f]) & = \text{(by axiom~\eqref{ax:type-subst-tm})} \\
ty_k(c_i)[f] & = \text{(since $Hom_{k,m}(c_1, \ldots c_m)$)} \\
ft^{m-i}_i(C)[c_1, \ldots c_{i-1}][f] & = \text{(by axiom~\eqref{ax:subst-subst})} \\
ft^{m-i}_i(C)[c_1[f], \ldots c_{i-1}[f]] &
\end{align*}

With these notations we can rewrite axioms \eqref{ax:type-subst-tm}, \eqref{ax:subst-var} and \eqref{ax:subst-subst} as follows:
\begin{align*}
ty_n(a[f]) & = A[f] \\
\text{ for each } f : B \to ft_k(A) & \text{, where } A = ty_k(a) \\
a[id_{ctx_{p,n}(a)}] & = a \\
a[g][f] & = a[g \circ f] \\
\text{ for each } f : C \to B \text{ and } & g : B \to ctx_{p,m}(a)
\end{align*}

Associativity of the composition follows from axiom~\eqref{ax:subst-subst}, and the fact that $id$ is identity for it follows from axioms \eqref{ax:subst-var} and \eqref{ax:var-subst}.

For every $A \in M_{(ty,k)}$ there is a bijection $\varphi$ between the set of $a \in M_{(tm,k)}$ such that $ty_k(a) = A$
    and the set of morphisms $f : ft_k(A) \to A$ such that $p_A \circ f = id_{ft_k(A)}$.
For every such $a \in M_{(tm,k)}$ we define $\varphi(a)$ as
\[ (v_{k,k-1}(ft_k(A)), \ldots v_{k,0}(ft_k(A)), a). \]
Note that if $(a_1, \ldots a_{k+1}) : B \to A$ is a morphism, then axiom~\eqref{ax:var-subst} implies that $p_A \circ (a_1, \ldots a_{k+1})$ equals to $(a_1, \ldots a_k)$.
Thus $\varphi(a)$ is a section of $p_A$.
Clearly, $\varphi$ is injective.
Let $f : ft_k(A) \to A$ be a section of $p_A$; then first $k$ components of $f$ must be identity on $ft_k(A)$.
So if $a$ is the last component of $f$, then $\varphi(a)$ equals to $f$.
Hence $\varphi$ is bijective.

If $A \in M_{(ty,k)}$, $B \in M_{(ctx,n)}$, and $f = (a_1, \ldots a_k) : B \to ft_k(A)$, then we define $f^*(A)$ as $A[f] = subst_{ty,n,k}(B, A, a_1, \ldots a_k)$.
Map $q(f,B)$ defined as the tuple with $i$-th component equals to
\[ \left\{
  \begin{array}{lr}
    a_i[v_{n+1,n}(A[f]), \ldots v_{n+1,1}(A[f])] & \text{ if } 1 \leq i \leq k \\
    v_{n+1,0}(A[f])                              & \text{ if } i = k+1
  \end{array}
\right. \]
Now we have the following commutative square:
\[ \xymatrix{ A[f] \ar[r]^-{q(f,A)} \ar[d]_{p_{A[f]}} & A \ar[d]^{p_A} \\
              B \ar[r]_-f                             & ft_k(A)
            } \]
We need to prove that this square is Cartesian.
By proposition~2.3 of \cite{c-systems} it is enough to construct a section $s_{f'} : B \to A[f]$ of $p_{A[f]}$
    for each $f' = (a_1, \ldots a_k, a_{k+1}) : B \to A$ and prove a few properties of $s_{f'}$.
We define $s_{f'}$ to be equal to $\varphi(a_{k+1})$.
Axioms \eqref{ax:var-subst} and \eqref{ax:subst-subst} implies that $q(f,B) \circ s_{f'} = f$.
To complete the proof that the square above is Cartesian we need, for every $g : ft_k(A) \to ft_m(C)$ and $A = C[g]$, prove that $s_{f'} = s_{q(g,C) \circ f'}$.
The last component of $q(g,C) \circ f'$ equals to $v_{n+1,0}(C[g])[f'] = a_{k+1}$.
Thus the last components of $q(g,C) \circ f'$ and $f'$ coincide, hence $s_{f'} = s_{q(g,C) \circ f'}$.

We are left to prove that operations $A[f]$ and $q(f,A)$ are functorial.
Equations $A[id_{ft_k(A)}] = A$ and $A[f \circ g] = A[f][g]$ are precisely axioms \eqref{ax:subst-var} and \eqref{ax:subst-subst}.
The fact that $q(id_{ft_k(A)}, A) = id_A$ follows from axiom~\ref{ax:var-subst}.
Now let $g : C \to B$ and $f : B \to ft_k(A)$ be morphisms; we need to show that $q(f \circ g, A) = q(f,A) \circ q(g,A[f])$.
The last component of $q(f,A) \circ q(g,A[f])$ equals to $v_{n+1,0}(A[f])[q(g,A[f])] = v_{m+1,0}(A[f][g])$,
    which equals to the last component of $q(f \circ g, A)$, namely $v_{m+1,0}(A[f \circ g])$.
If $1 \leq i \leq k$, then $i$-th component of $q(f,A) \circ q(g,A[f])$ equals to
\[ a_i[v_{n+1,n}(A[f]), \ldots v_{n+1,1}(A[f])][q(g,A[f])] = a_i[b_1', \ldots b_n'] \]
where $a_i$ is $i$-th component of $f$, $b_i$ is $i$-th component of $g$, and $b_i'$ equals to $b_i[v_{m+1,m}(A[f][g]), \ldots v_{m+1,1}(A[f][g])]$.
$i$-th component of $q(f \circ g, A)$ equals to
\[ a_i[g][v_{m+1,m}(A[f \circ g]), \ldots v_{m+1,1}(A[f \circ g])] = a_i[b_1'', \ldots b_n''], \]
where $b_i'' = b_i[v_{m+1,m}(A[f \circ g]), \ldots v_{m+1,1}(A[f \circ g])]$.
Thus $q(f \circ g, A) = q(f,A) \circ q(g,A[f])$.
This completes the construction of contextual category $F(M)$.

\begin{prop}[T1-CCat]
$F$ is functorial, and functor $F : \Mod{\substTh} \to \ccat$ is an equivalence of categories.
\end{prop}
\begin{proof}
Given a map of $\substTh$ models $\alpha : M \to N$, we define a map of contextual categories $F(\alpha) : F(M) \to F(N)$.
$F(\alpha)$ is already defined on objects.
Let $f = (a_1, \ldots a_k) \in Hom_{n,k}(B,A)$.
We define $F(\alpha)(f)$ as $(\alpha(a_1), \ldots \alpha(a_k)) \in Hom_{n,k}(\alpha(B), \alpha(A))$.
$F(\alpha)$ preserves identity morphisms, compositions, $f^*(A)$, and $q(f,A)$ since all of these operations are defined in terms of $\substTh$ operations.
Clearly, $F$ preserves identity maps and compositions of maps of $\substTh$ models.
Thus $F$ is a functor.

First, note that if $a \in M_{(tm,k)}$ and $\alpha : M \to N$, then $F(\alpha)(\varphi(a)) = \varphi(\alpha(a))$.
Indeed, consider the following sequence of equations:
\begin{align*}
F(\alpha)(\varphi(a)) & = \\
F(\alpha)(v_{k,k-1}(ctx_{tm,k}(a)), \ldots v_{k,0}(ctx_{tm,k}(a)), a) & = \\
(v_{k,k-1}(ctx_{tm,k}(\alpha(a))), \ldots v_{k,0}(ctx_{tm,k}(\alpha(a))), \alpha(a)) & = \\
\varphi(\alpha(a)) & .
\end{align*}

Now, we prove that $F$ is faithful.
Let $\alpha,\beta : M \to N$ be a pair of maps of $\substTh$ models such that $F(\alpha) = F(\beta)$.
Then $\alpha$ and $\beta$ coincide on contexts.
Given $a \in M_{(tm,n)}$ we have the following equation: $\alpha(a) = \varphi^{-1}(F(\alpha)(\varphi(a))) = \varphi^{-1}(F(\beta)(\varphi(a))) = \beta(a)$.

Now, we prove that $F$ is full.
Let $\alpha : F(M) \to F(N)$ be a map of contextual categories.
Then we need to define $\beta : M \to N$ such that $F(\beta) = \alpha$.
If $A \in M_{(ctx,n)}$, then we let $\beta(A) = \alpha(A)$.
Note that if $f : ft_n(A) \to A$ is a section of $p_A$, then $\alpha(f)$ is a section of $\alpha(A)$.
If $a \in M_{(tm,n)}$, then we let $\beta(a) = \varphi^{-1}(\alpha(\varphi(a)))$.

Maps $F(\beta)$ and $\alpha$ agree on contexts.
We prove by induction on $k$ that they coincide on morphisms $f = (a_1, \ldots a_k) \in M(Hom_{n,k})(B,A)$.
If $k = 0$, then $F(A)$ is terminal objects, hence $F(\beta) = \alpha$.
Suppose $k > 0$ and consider the following equation: $f = q((a_1, \ldots a_{k-1}), A) \circ \varphi(a_k)$.
By induction hypothesis we know that $F(\beta)(q((a_1, \ldots a_{k-1}), A)) = \alpha(q((a_1, \ldots a_{k-1}), A))$.
Thus we only need to prove that $F(\beta)(\varphi(a_k)) = \alpha(\varphi(a_k))$.
But $F(\beta)(\varphi(a_k)) = \varphi(\beta(a_k)) = \varphi(\varphi^{-1}(\alpha(\varphi(a_k)))) = \alpha(\varphi(a_k))$.

Finally, we prove that $F$ is essentially surjective on objects.
Given contextual category $C$ we define $\substTh$ model $M$.
Let $M_{(ctx,n)}$ be equal to $Ob_n(C)$ and $M_{(tm,n)}$ be the set of pairs of objects $A \in Ob_{n+1}(C)$ and sections of $p_A : A \to ft_n(A)$.
Let $ty_n$ be the obvious projection.
We will usually identify $a \in M_{(tm,n)}$ with the section $ctx_{tm,n}(a) \to ty_n(a)$.

For each $n,k \in \mathbb{N}$ we define partial function
\[ subst_{ty,n,k} : M_{(ctx,n)} \times M_{(ty,k)} \times M_{(tm,n)}^k \to M_{(ty,n)} \]
such that $ft_n(subst_{ty,n,k}(B, A, a_1, \ldots a_k)) = B$.
We also define morphism
\[ q_{n,k} \in Hom_{n+1,k}(subst_{ty,n,k}(B, A, a_1, \ldots a_k), A) \]
whenever $subst_{ty,n,k}(B, A, a_1, \ldots a_k)$ is defined.
We define $subst_{ty,n,k}$ and $q_{n,k}$ by induction on $k$.
Let $subst_{ty,n,0}(B,A) = !_B^*(A)$ and $q_{n,0} = q(!_B,A)$ where $!_B : B \to Ob_0(C)$ is the unique morphism.
\[ \xymatrix{ subst_{ty,n,0}(B,A) \ar[r]^-{q_{n,0}} \ar[d] \pb & A \ar[d]^{p_A} \\
              B \ar[r]_{!_B} & 1
            } \]

Let $subst_{ty,n,k+1}(B, A, a_1, \ldots a_{k+1})$ be defined whenever $subst_{ty,n,k}(B, ft_k(A), \allowbreak a_1, \ldots a_k)$ is defined
    and $ty_n(a_{k+1}) = subst_{ty,n,k}(B, ft_k(A), a_1, \ldots a_k)$.
In this case we let $subst_{ty,n,k+1}(B, A, a_1, \ldots a_{k+1}) = f^*(A)$ and $q_{n,k+1} = q(f,A)$ where $f$ is the composition of $a_{k+1}$ and $q_{n,k}$.
\[ \xymatrix{ subst_{ty,n,k+1}(B, A, a_1, \ldots a_{k+1}) \ar[rr]^-{q_{n,k+1}} \ar[d] \pb & & A \ar[d]^{p_A} \\
              B \ar[r]_-{a_{k+1}} & ty_n(a_{k+1}) \ar[r]_-{q_{n,k}} & ft_k(A)
            } \]
It is easy to see by induction on $k$ that axiom~\eqref{ax:def-subst} holds.
Axiom~\eqref{ax:type-subst-ty} holds by definition of $subst_{ty,n,k}$.

The definition of predicates $Hom_{n,k}$ makes sense in $M$ now.
Thus we can define as before the set $Hom^M_{n,k}(B,A)$ of morphisms in $M$ as the set of tuples $(a_1, \ldots a_k)$ such that $Hom_{n,k}(B, A, a_1, \ldots a_k)$.
There is a bijection $\alpha : Hom^M_{n,k}(B,A) \to Hom_{n,k}(B,A)$ such that
    $subst_{ty,n,k}(B, A, a_1, \ldots a_k) = \alpha(a_1, \ldots a_k)^*(A)$ and $q_{n,k} = q(\alpha(a_1, \ldots a_k), A)$.
We define $\alpha$ by induction on $k$.
Both $Hom^M_{n,0}(B,A)$ and $Hom_{n,0}(B,A)$ are singletons, so there is a unique bijection between them.
If $(a_1, \ldots a_k) \in Hom^M_{n,k}(B,ft_k(A))$, then there is a bijection between morphisms $f \in Hom_{n,k+1}(B,A)$
    satisfying $p_A \circ f = \alpha(a_1, \ldots a_k)$ and sections of $p_{\alpha(a_1, \ldots a_k)^*(A)}$.
By induction hypothesis these sections are just sections of $p_{subst_{ty,n,k}(B, A, a_1, \ldots a_k)}$.
This gives us a bijection between $Hom^M_{n,k+1}(B,A)$ and $Hom_{n,k+1}(B,A)$, namely $\alpha(a_1, \ldots a_{k+1}) = q(\alpha(a_1, \ldots a_k), A) \circ a_{k+1}$.
Then the required equations hold by definition.

Now, we define total functions $v_{n,i} : M_{(ctx,n)} \to M_{(tm,n)}$.
Let $v_{n,i}(A)$ be equal to $(p^{i+1}(A)^*(ft^i_{n-i}(A)), s_{p^i_A})$.
\[ \xymatrix{ p^{i+1}(A)^*(ft^i_{n-i}(A)) \ar[r] \ar[d] \pb & ft^i_{n-i}(A) \ar[d]^{p_{ft^i_{n-i}(A)}} \\
              A \ar[r]_{p^{i+1}(A)} \ar@/^1pc/[u]^{s_{p^i_A}} \ar[ur]_{p^i_A} & ft^{i+1}_{n-i-1}(A)
            } \]
Axiom~\eqref{ax:def-var} holds by definition.
By induction on $n-i$ it is easy to see that $\alpha(v_{n,n-1}(A), \ldots v_{n,i}(A))$ equals to $p_A^i : A \to ft^i_{n-i}(A)$.
Axiom~\eqref{ax:type-var} follows from the following sequence of equations:
\begin{align*}
subst_{ty,n,n-i-1}(A, ft^i_{n-i}(A), v_{n,n-1}(A), \ldots v_{n,i+1}(A)) & = \\
\alpha(v_{n,n-1}(A), \ldots v_{n,i+1}(A))^*(ft^i_{n-i}(A)) & = \\
p^{i+1}(A)^*(ft^i_{n-i}(A)) & = \\
ty_n(v_{n,i}(A)) & .
\end{align*}
Axiom~\eqref{ax:subst-var} follows from the facts that $\alpha(v_{n,n-1}(ft_n(A)), \ldots v_{n,0}(ft_n(A))) = id_{ft_n(A)}$ and $id_{ft_n(A)}^*(A) = A$.

Now, we define partial functions $subst_{tm,n,k} : M_{(ctx,n)} \times M_{(tm,k)} \times M_{(tm,n)}^k \to M_{(tm,n)}$.
Function $subst_{tm,n,k}(B, a, a_1, \ldots a_k)$ is defined whenever \[ Hom_{n,k}(B, ctx_{tm,k}(a), a_1, \ldots a_k) \] holds.
In this case we let $subst_{tm,n,k}(B, a, a_1, \ldots a_k) = a[\alpha(a_1, \ldots a_k)]$ where $a[f] = s_{a \circ f}$.
Axioms \eqref{ax:def-subst} and \eqref{ax:type-subst-tm} hold by definition.
Axiom~\eqref{ax:subst-var} follows from the fact that $id_{ctx_{tm,n}(a)}^*(a) = a$.

To prove axiom~\eqref{ax:var-subst} note that $p_A \circ \alpha(a_1, \ldots a_{k+1}) = \alpha(a_1, \ldots a_k)$ by definition of $\alpha$.
Hence $p^i(A) \circ \alpha(a_1, \ldots a_k) = \alpha(a_1, \ldots a_{k-i})$.
Also note that $s_{\alpha(a_1, \ldots a_k)} = a_k$.
Now the axiom follows from the following equations:
\begin{align*}
subst_{tm,n,k}(B, v_{k,i}(A), a_1, \ldots a_k) & = \\
s_{v_{k,i}(A) \circ \alpha(a_1, \ldots a_k)} & = \\
s_{q(p^{i+1}(A), ft^i_{n-i}(A)) \circ v_{k,i}(A) \circ \alpha(a_1, \ldots a_k)} & = \\
s_{p^i(A) \circ \alpha(a_1, \ldots a_k)} & = \\
s_{\alpha(a_1, \ldots a_{k-i})} & = \\
a_{k-i} & .
\end{align*}

Now, we prove that $\alpha$ preserves compositions.
To do this we need to show that $\alpha(a_1, \ldots a_k) \circ f = \alpha(a_1[f], \ldots a_k[f])$.
We do this by induction on $k$.
For $k = 0$ it is trivial and for $k > 0$ we have the following sequence of equations:
\begin{align*}
\alpha(a_1, \ldots a_k) \circ f & = \\
q(\alpha(a_1, \ldots a_{k-1}), A) \circ a_k \circ f & = \\
q(\alpha(a_1, \ldots a_{k-1}), A) \circ q(f, B[\alpha(a_1, \ldots a_k)]) \circ a_k[f] & = \\
q(\alpha(a_1, \ldots a_{k-1}) \circ f, A) \circ a_k[f] & = \\
q(\alpha(a_1[f], \ldots a_{k-1}[f]), A) \circ a_k[f] & = \\
\alpha(a_1[f], \ldots a_k[f]) & .
\end{align*}

Now, axiom \eqref{ax:subst-subst} follows from the facts that $\alpha$ preserves compositions and $(f \circ g)^*(A) = f^*(g^*(A))$.
This completes the construction of $\substTh$ model $M$ from a contextual category $C$.
To finish the proof we need to show that $F(M)$ is isomorphic to $C$.
The isomorphism is given by bijection $\alpha$.
We already saw that $\alpha$ preserves the structure of contextual categories.
Thus $\alpha$ is a morphism of contextual categories, and it is easy to see that $\alpha^{-1}$ also preserves the structure.
Hence $\alpha$ is isomorphism and $F$ is an equivalence.
\end{proof}

Let $u : \substTh \to \mathbb{T}$ be an algebraic dependent type theory with substitution.
Then it follows from \rprop{func-mod} and \rprop{T1-CCat} that models of $\mathbb{T}$ are contextual categories with additional structure,
    where $u^* : \Mod{\mathbb{T}} \to \Mod{\substTh}$ is the forgetful functor.

\section{Algebraic dependent type theories}

In this section we consider partial Horn theories with additional structure which we call \emph{stable}.
We also define the category $\algtt$ of algebraic dependent type theories and give a few examples of such theories.

\subsection{Stable theories}

First, let us define prestable theories.
For every set $\mathcal{S}_0$, we define the corresponding set $\mathcal{S}$ of sorts as $\mathcal{S}_0 \times \mathbb{N}$.
We call elements of $\mathcal{S}_0$ \emph{basic sorts}.
Suppose that $\mathcal{S}_0$ contains a distinguished sort $ctx$.
Let $\mathbb{T}_{\mathcal{S}_0}$ be a theory with the following function symbols:
\begin{align*}
\emptyCtx  &: (ctx,0) \\
ft_n & : (ctx,n+1) \to (ctx,n) \\
ctx_{p,n} & : (p,n) \to (ctx,n) \text{ for every } p \in \mathcal{S}_0
\end{align*}
and the following axioms:
\begin{align*}
& \sststile{}{} \emptyCtx\!\downarrow \\
& \sststile{}{x} x = \emptyCtx \\
& \sststile{}{x} ctx_{ctx,n}(x) = x
\end{align*}

To define prestable theories, we need to introduce a few auxiliary constructions.
First, we define a function $L : \mathcal{C} \to \mathcal{C}$ as follows:
\begin{align*}
L(ctx,n) & = L(ctx,n+1) \\
L(tm,n) & = L(tm,n+1)
\end{align*}

For every set $\mathcal{F}$ of function symbols, we define another set $L(\mathcal{F})$ which consists of symbols $L(\sigma)$ for every $\sigma \in \mathcal{F}$.
If $\sigma : s_1 \times \ldots \times s_k \to s$, then $L(\sigma) : (ctx,1) \times L(s_1) \times \ldots \times L(s_k) \to L(s)$.
For every set of variables $V$ we define a set $L(V)$ which contains a variable $x$ of sort $L(s)$ for every variable $x$ of sort $s$ in $V$.
For every terms $\Gamma \in Term_{L(\mathcal{F})}(L(V))_{(ctx,1)}$ and $t \in Term_{\mathcal{F}}(V)_{(p,n)}$,
we define a restricted term $L(\Gamma,t) \in RTerm_{L(\mathcal{F})}(L(V))_{(p,n+1)}$ as follows:
\begin{align*}
L(\Gamma, x) & = x|_{L(ctx_{p,n})(\Gamma, x) \downarrow} \\
L(\Gamma, \sigma(t_1, \ldots t_k)) & = L(\sigma)(\Gamma, L(\Gamma, t_1), \ldots L(\Gamma, t_k))
\end{align*}

For every set $\mathcal{P}$ of relation symbols, we define set $L(\mathcal{P})$ which consists of symbols
    $L(R) : (ctx,1) \times L(s_1) \times \ldots \times L(s_k)$ for every $R \in \mathcal{P}$, $R : s_1 \times \ldots \times s_k$.
For every formula $\varphi \in Form_\mathcal{P}(V)$ and term $\Gamma \in Term_{L(\mathcal{F})}(L(V))_{(ctx,1)}$,
we define a formula $L(\Gamma, \varphi) \in Form_{L(\mathcal{P})}(L(V))$ as follows:
\begin{align*}
L(\Gamma, t_1 = t_2) & = (L(\Gamma, t_1) = L(\Gamma, t_2)) \\
L(\Gamma, R(t_1, \ldots t_k)) & = L(R)(\Gamma, L(\Gamma, t_1), \ldots L(\Gamma, t_k))
\end{align*}

Now, let us define a functor $L : \mathbb{T}_{\mathcal{S}_0}/\Th_\mathcal{S} \to \mathbb{T}_{\mathcal{S}_0}/\Th_\mathcal{S}$.
Let $L((\mathcal{S}, \mathcal{F}, \mathcal{P}), \mathcal{A}) = ((\mathcal{S}, L(\mathcal{F}) \cup \mathcal{F}_{\mathcal{S}_0}, L(\mathcal{P})), \mathcal{A}' \cup \mathcal{A}_{\mathcal{S}_0})$,
where $\mathcal{F}_{\mathcal{S}_0}$ and $\mathcal{A}_{\mathcal{S}_0}$ are the sets of function symbols and axioms of $\mathbb{T}_{\mathcal{S}_0}$, and $\mathcal{A}'$ consists of the following axioms:
\[ ft^n(ctx_{p,n+1}(x)) = \Gamma \sststile{}{\Gamma,x} ctx_{p,n+1}(x) = L(ctx_{p,n})(\Gamma,x) \]
for every $p \in \mathcal{S}_0$,
\begin{align*}
L(\sigma)(\Gamma, x_1, \ldots x_k)\!\downarrow & \sststile{}{\Gamma, x_1, \ldots x_k} ft^n(ctx_{p,n}(L(\sigma)(\Gamma, x_1, \ldots x_k))) = \Gamma \\
L(\sigma)(\Gamma, x_1, \ldots x_k)\!\downarrow & \sststile{}{\Gamma, x_1, \ldots x_k} ft^{n_i}(ctx_{p_i,n_i}(x_i)) = \Gamma
\end{align*}
for every $\sigma \in \mathcal{F}$, $\sigma : (p_1,n_1) \times \ldots \times (p_k,n_k) \to (p,n)$ and every $1 \leq i \leq k$,
\[ L(R)(\Gamma, x_1, \ldots x_k) \sststile{}{\Gamma, x_1, \ldots x_k} ft^{n_i}(ctx_{p_i,n_i}(x_1)) = \Gamma \]
for every $R \in \mathcal{P}$, $R : (p_1,n_1) \times \ldots \times (p_k,n_k)$ and every $1 \leq i \leq k$.

If $f : \mathbb{T} \to \mathbb{T}'$, then let $L(f) : L(\mathbb{T}) \to L(\mathbb{T}')$ be defined as follows:
\begin{align*}
L(f)(L(\sigma)(\Gamma, x_1, \ldots x_k)) & = L(\Gamma, f(\sigma(x_1, \ldots x_k))) \\
L(f)(L(R)(\Gamma, x_1, \ldots x_k)) & = L(\Gamma, f(R(x_1, \ldots x_k)))
\end{align*}
It is easy to see that this defines a morphism of theories and that $L$ preserves identity morphisms and compositions.

\begin{defn}
A \emph{prestable (essentially) algebraic theory} is an algebra for functor $L$,
that is a pair $(\mathbb{T},\alpha)$, where $\mathbb{T}$ is a theory under $\mathbb{T}_{\mathcal{S}_0}$ and $\alpha : L(\mathbb{T}) \to \mathbb{T}$.
The category $\PSt_{\mathcal{S}_0}$ of prestable theories is the category of algebras for $L$.
\end{defn}

\begin{defn}
A prestable theory is called \emph{stable} if the following theorem holds for every axiom $\varphi \sststile{}{x_1 : (p_1,n_1), \ldots x_k : (p_k,n_k)} \psi$ in $\mathcal{A}$:
\[ \alpha L(\Gamma,\varphi) \land \bigwedge_{1 \leq i \leq k} ft^{n_i}(ctx_{p_i,n_i}(x_i)) = \Gamma \sststile{}{\Gamma, x_1, \ldots x_k} \alpha L(\Gamma,\psi). \]
The category of stable theories is denoted by $\St_{\mathcal{S}_0}$.

Let $c$ be the prestable theory generated by a single constant $c : (ctx,1)$.
Then a prestable theory under $c$ is called \emph{$c$-stable} if the following sequents are derivable:
\begin{align*}
\alpha L(\sigma)(\Gamma, x_1, \ldots x_k)\!\downarrow & \sststile{}{\Gamma, x_1, \ldots x_k} \Gamma = c \\
\alpha L(R)(\Gamma, x_1, \ldots x_k) & \sststile{}{\Gamma, x_1, \ldots x_k} \Gamma = c \\
\alpha L(c,\varphi) \land \bigwedge_{1 \leq i \leq k} ft^{n_i}(ctx_{p_i,n_i}(x_i)) = c & \sststile{}{x_1, \ldots x_k} \alpha L(c,\psi)
\end{align*}
for every function symbol $\sigma$, every predicate symbol $R$, and every axiom $\varphi \sststile{}{x_1, \ldots x_k} \psi$.
The category of $c$-stable theories is denoted by $\cSt_{\mathcal{S}_0}$.
\end{defn}

The theory of substitutions is stable.
Indeed, we can define maps $\alpha : L(\substTh) \to \substTh$ as follows:
\begin{align*}
\alpha(L(ty_n)(\Gamma,a)) & = ty_{n+1}(a)|_{ft^n(ctx_{tm,n+1}(a)) = \Gamma} \\
\alpha(L(v_{n,i})(\Gamma,\Delta)) & = v_{n+1,i}(\Delta)|_{ft^n(\Delta) = \Gamma}
\end{align*}
and $\alpha(L(subst_{p,n,k})(\Gamma, \Delta, B, a_1, \ldots a_k))$ is defined as
\[ subst_{p,n+1,k+1}(\Delta, B, v_{n+1,n}(\Delta), a_1, \ldots a_k)|_{ft^n(\Delta) = \Gamma} \]

The construction of colimits in \rprop{th-cocomplete} implies that $L$ preserves colimits.
It follows that $\PSt_{\mathcal{S}_0}$ is cocomplete.
The categories of stable and $c$-stable theories are closed in $\PSt_{\mathcal{S}_0}$ under colimits.

\subsection{Contextual theories}

The definition of prestable theories has a disadvantage that terms contain a lot of redundant information.
For example, when we describe a term we need to repeat the context in which it is defined several times.
The following notion allows us to omit this redundant information as we discuss below.

\begin{defn}
Let $\mathbb{T}_b$ be a prestable theory.
A \emph{contextual theory under $\mathbb{T}_b$} is a prestable theory $\mathbb{T}$ such that the following conditions hold:
\begin{enumerate}
\item There exists a set of function symbols $\mathcal{F}_0$ (which we call \emph{basic function symbols}) such that
the set of function symbols of $\mathbb{T}$ consists of function symbols of $\mathbb{T}_b$ together with symbols
\[ \sigma_m : (ctx,m) \times (p_1,n_1+m) \times \ldots \times (p_k,n_k+m) \to (p,n+m) \]
for every $\sigma : (p_1,n_1) \times \ldots \times (p_k,n_k) \to (p,n) \in \mathcal{F}_0$ and $m \in \mathbb{N}$.
Moreover, if $\sigma : s_1 \times \ldots \times s_k \to s \in \mathcal{F}_0$, then $s \neq (ctx,0)$.
\item There exists a set of predicate symbols $\mathcal{P}_0$ (which we call \emph{basic predicate symbols}) such that
the set of predicate symbols of $\mathbb{T}$ consists of predicate symbols of $\mathbb{T}_b$ together with symbols
\[ R_m : (ctx,m) \times (p_1,n_1+m) \times \ldots \times (p_k,n_k+m) \]
for every $R : (p_1,n_1) \times \ldots \times (p_k,n_k) \in \mathcal{P}_0$ and $m \in \mathbb{N}$.
\item Every axiom of $\mathbb{T}_b$ is an axiom of $\mathbb{T}$.
\item $\alpha_\mathbb{T} : L(\mathbb{T}) \to \mathbb{T}$ is defined as follows:
\begin{align*}
\alpha_\mathbb{T}(L(\sigma_m)(\Gamma, \Delta, x_1, \ldots x_k)) & = \sigma_{m+1}(\Delta, x_1, \ldots x_k)|_{ctx^{n+m}(\Delta) = \Gamma} \\
\alpha_\mathbb{T}(L(R_m)(\Gamma, \Delta, x_1, \ldots x_k)) & = R_{m+1}(\Delta, x_1, \ldots x_k) \land ctx^{n+m}(\Delta) = \Gamma
\end{align*}
and for every symbol of $\mathbb{T}_b$, it is defined in the same way as in $\mathbb{T}_b$.
\end{enumerate}
\end{defn}

Since we can always infer the index $m$ for every function symbol $\sigma_m$ if we know its sort, we usually omit this index.
To specify the omitted argument, we use the following syntax: $\Gamma \vdash t$,
which stands for $\sigma(\Gamma, t_1, \ldots t_k)$ if $t = \sigma(t_1, \ldots t_k)$ and for $x|_{ctx(x) = \Gamma}$ if $t = x$.
Of course, if some arguments are omitted in $\Gamma$, then we need to know its context too in order to infer them.
Thus, we may write $A_1, \ldots A_n \vdash t$ which stands for $(\ldots ((\emptyCtx \vdash A_1) \vdash A_2) \ldots \vdash A_n) \vdash t$.
We also use this notation in formulas: $\Gamma \vdash t \equiv t'$ stands for $(\Gamma \vdash t) = (\Gamma \vdash t')$
and $\Gamma \vdash R(t_1, \ldots t_k)$ stands for $R(\Gamma, (\Gamma \vdash t_1), \ldots (\Gamma \vdash t_k))$.

Also, we use the standard notation: $\Gamma \vdash A\ type$ stands for $\Gamma \vdash A\!\downarrow$ if $A : (ty,n)$ and $\Gamma \vdash a : A$ stands for $ty(\Gamma \vdash a) = (\Gamma \vdash A)$.
Sequents $\varphi_1 \land \ldots \land \varphi_n \sststile{}{V} \psi$ and $\varphi_1 \land \ldots \land \varphi_n \ssststile{}{V} \psi$ are written as
\medskip
\begin{center}
\AxiomC{$\varphi_1$}
\AxiomC{$\ldots$}
\AxiomC{$\varphi_n$}
\TrinaryInfC{$\psi$}
\DisplayProof
\qquad
and
\qquad
\AxiomC{$\varphi_1$}
\AxiomC{$\ldots$}
\AxiomC{$\varphi_n$}
\doubleLine
\RightLabel{,}
\TrinaryInfC{$\psi$}
\DisplayProof
\end{center}
respectively.

Finally, we use the following syntax:
\[ \sigma(A^1_1, \ldots A^1_{n_1}.\ b_1, \ldots A^k_1, \ldots A^k_{n_k}.\ b_k) \]
for a term of sort $(p,m+n)$ in a contextual theory, where $\sigma : (p_1,n_1) \times \ldots \times (p_k,n_k) \to (p,n)$,
$b_i$ is a term of sort $(p_i,m+n_i)$, and $A^i_j$ is a term of sort $(ty,m+j-1)$.
The expression $\Gamma \vdash \sigma(A^1_1, \ldots A^1_{n_1}.\ b_1, \ldots A^k, \ldots A^k_{n_k}.\ b_k)$ stands for
\[ \sigma_m(\Gamma, (\Gamma, A^1_1, \ldots A^1_{n_1} \vdash b_1), \ldots (\Gamma, A^k_1, \ldots A^k_{n_k} \vdash b_k)). \]
Of course, if some $b_i$ is a variable, then we can omit $A^i_1, \ldots A^i_{n_i}$.
We also can omit this context if there is a theorem of the following form: 
\[ E \vdash \sigma_m(x_1, \ldots x_k)\!\downarrow\ \sststile{}{E, x_1, \ldots x_k} E \vdash ctx(x_i) \equiv \Delta \]
for some $\Delta$ such that $x_i \notin FV(\Delta)$.
Then $A^i_1, \ldots A^i_{n_i}$ must be equal to $((\Gamma \vdash ft^{n_i-1}(\Delta)), \ldots (\Gamma \vdash \Delta))[\rho]$,
where $\rho(E) = \Gamma$ and $\rho(x_j) = (\Gamma, A^j_1, \ldots A^j_{n_j} \vdash b_j)$.

The following lemma shows that we can always replace a prestable theory with a contextual one.

\begin{lem}[stable-con]
Let $\mathbb{T}_b$ be a prestable theory.
Every prestable theory under $\mathbb{T}_b$ is isomorphic to a contextual theory under $\mathbb{T}_b$.
\end{lem}
\begin{proof}
Let $\mathbb{T}$ be a prestable theory together with a map $f : \mathbb{T}_b \to \mathbb{T}$ with $\mathcal{F}_0$ and $\mathcal{P}_0$ as the sets of function and predicate symbols.
First, note that we may assume that for every $\sigma : s_1 \times \ldots \times s_k \to s$ in $\mathcal{F}_0$, $s \neq (ctx,0)$.
Indeed, we can always replace such a function symbol with a predicate symbol $R_\sigma : s_1 \times \ldots \times s_k$.

Second, note that for every term $t \in Term_{\mathcal{F}_0}(V)_{(p,n)}$ and every $m \in \mathbb{N}$, we can construct the following restricted term:
\[ \alpha L(ft^{m-1}(\Gamma), \alpha L(ft^{m-2}(\Gamma), \ldots \alpha L(\Gamma, t))) \]
in $RTerm_{\mathbb{T}}(L^m(V) \amalg \{ \Gamma : (ctx,m) \})_{(p,n+m)}$, which we denote by $\Gamma \times t$.
Analogously, we can define for every formula $\varphi \in Form_{\mathbb{T}}(V)$ and every $m \in \mathbb{N}$,
a formula $\Gamma \times \varphi \in Form_{\mathbb{T}}(L^m(V) \amalg \{ \Gamma : (ctx,m) \})$.

Let $\mathbb{T}'$ be a contextual theory under $\mathbb{T}_b$ defined from the sets $\mathcal{F}_0$ and $\mathcal{P}_0$.
Note that every term (formula, sequent) of $\mathbb{T}$ is naturally a term (formula, sequent) of $\mathbb{T}'$.
Axioms of $\mathbb{T}'$ is the axioms of $\mathbb{T}$ together with the following axioms:
\begin{align*}
& \sststile{}{x_1, \ldots x_k} \tau_0(\emptyCtx, x_1, \ldots x_k) \cong f(\tau(x_1, \ldots x_k)) \\
P_0(\emptyCtx, x_1, \ldots x_k) & \ssststile{}{x_1, \ldots x_k} f(P(x_1, \ldots x_k)) \\
& \sststile{}{\Gamma, x_1, \ldots x_k} \Gamma \times \sigma_0(\emptyCtx, x_1, \ldots x_k) \cong \sigma_m(\Gamma, x_1, \ldots x_k) \\
\Gamma \times R_0(\emptyCtx, x_1, \ldots x_k) & \ssststile{}{\Gamma, x_1, \ldots x_k} R_m(\Gamma, x_1, \ldots x_k)
\end{align*}
for every function symbol $\tau$ and predicate symbol $P$ of $\mathbb{T}_b$ and every $\sigma \in \mathcal{F}_0$ and $R \in \mathcal{P}_0$.

There is an obvious map $\mathbb{T} \to \mathbb{T}'$ and we can define a map $\mathbb{T}' \to \mathbb{T}$
which maps $\sigma_m(\Gamma, x_1, \ldots x_k)$ to $\Gamma \times \sigma_0(\emptyCtx, x_1, \ldots x_k)$,
$R_m(\Gamma, x_1, \ldots x_k)$ to $\Gamma \times R_0(\emptyCtx, x_1, \ldots x_k)$,
$\tau_0(\Gamma, x_1, \ldots x_k)$ to $f(\tau(x_1, \ldots x_k))|_{\Gamma\downarrow}$, and $P_0(\Gamma, x_1, \ldots x_k)$ to $f(P(x_1, \ldots x_k))|_{\Gamma\downarrow}$.
Axioms guarantee that these maps are inverses of each other.
\end{proof}

Contextual theories constructed in the previous lemma are not convenient in practice, but usually theories are defined in a contextual form.
It is easy to define such theory: we just need to specify sets $\mathcal{F}_0$ and $\mathcal{P}_0$ and the set of axioms.
It is also easy to define a morphism of contextual theories since we only need to define it on symbols from $\mathcal{F}_0$ and $\mathcal{P}_0$.
Then it uniquely extends to a morphism of prestable theories.

\subsection{Algebraic dependent type theories}

Algebraic dependent type theories are prestable theories under $\substTh$ in which substitution commutes with all function symbols.
To define such theories, we need to define weakening first.
For every $p \in \{ty,tm\}$, the operations of weakening $wk^{m,l}_{p,n} : (ctx,n+m) \times (p,n+l) \to (p,n+m+l)$ are defined as follows:
\begin{align*}
wk^{m,0}_{p,n}(\Gamma,a) & = subst_{p,n+m,n}(\Gamma, a, v_{n+m-1}, \ldots v_m) \\
wk^{m,l+1}_{p,n}(\Gamma,a) & = subst_{p,n+m+l+1,n+l+1}(\Gamma', a, v_{n+m+l}, \ldots v_{m+l+1}, v_l, \ldots v_0),
\end{align*}
where $\Gamma' = wk^{m,l}_{ty,n}(\Gamma,ctx(a))$.
We also define $wk^{m,l}_{ctx,n} : (ctx,n+m) \times (ctx,n+l) \to (ctx,n+m+l)$ as follows:
\begin{align*}
wk^{m,0}_{ctx,n}(\Gamma,a) & = \Gamma \\
wk^{m,l+1}_{ctx,n}(\Gamma,a) & = wk^{m,l}_{ty,n}(\Gamma,a).
\end{align*}

Now, we need to introduce a new derived operation.
For every $m,n,k \in \mathbb{N}$ and $p \in \{ ctx, ty, tm \}$, we define the following function:
\[ subst^m_{p,n,k} : (ctx,n) \times (p,k+m) \times (tm,n)^k \to (p,n+m). \]
First, let $subst^0_{ctx,n,k}(B, A, a_1, \ldots a_k) = B$ and $subst^{m+1}_{ctx,n,k} = subst^m_{ty,n,k}$.
If $p \in \{ ty, tm \}$, then let $subst^m_{p,n,k}(B, a, a_1, \ldots a_k)$ be equal to
\[ subst_{p,n+m,k+m}(B', a, wk^{m,0}_{tm,n}(a_1), \ldots wk^{m,0}_{tm,n}(a_k), v_{m-1}, \ldots v_0), \]
where $B' = subst^m_{ctx,n,k}(B, ctx_{k+m}(a), a_1, \ldots a_k)$.

\begin{defn}
A prestable theory under $\substTh$ is an \emph{algebraic dependent type theory} if,
for every $\sigma \in \mathcal{F}$, $\sigma : (p_1,n_1) \times \ldots \times (p_k,n_k) \to (p,n)$
and every $R \in \mathcal{P}$, $R : (p_1,n_1) \times \ldots \times (p_k,n_k)$, the following sequents are derivable in it:
\medskip
\begin{center}
\AxiomC{$\Delta \times \sigma(b_1, \ldots b_k) \downarrow$}
\AxiomC{$\bigwedge_{1 \leq i \leq m} ty(a_i) = subst_{ty,l,i-1}(\Gamma, ft^{m-i}(\Delta), a_1, \ldots a_{i-1})$}
\BinaryInfC{$subst^n_{p,l,m}(\Gamma, \Delta \times \sigma(b_1, \ldots b_k), a_1, \ldots a_m) = \Gamma \times \sigma(b_1', \ldots b_k')$}
\DisplayProof
\end{center}
\medskip

\begin{center}
\AxiomC{$\Delta \times R(b_1, \ldots b_k)$}
\AxiomC{$\bigwedge_{1 \leq i \leq m} ty(a_i) = subst_{ty,l,i-1}(\Gamma, ft^{m-i}(\Delta), a_1, \ldots a_{i-1})$}
\BinaryInfC{$\Gamma \times R(b_1', \ldots b_k')$}
\DisplayProof
\end{center}
\medskip
where $b_i' = subst^{n_i}_{p_i,l,m}(\Gamma, b_i, a_1, \ldots a_m)$.

The category of algebraic dependent type theories will be denoted by $\algtt$.
\end{defn}

The construction of colimits in \rprop{th-cocomplete} implies that $\algtt$ is closed under colimits in $\substTh/\PSt_\mathcal{C}$.
The inclusion functor $\algtt \to \substTh/\PSt_\mathcal{C}$ has a left adjoint $\substTh/\PSt_\mathcal{C} \to \algtt$, which simply adds the required axioms.

We can prove a stronger version of \rlem{stable-con} for algebraic dependent type theories:
\begin{lem}[adtt-con]
Every algebraic dependent type theory is isomorphic to a contextual theory in which every function symbol in $\mathcal{F}_0$ has a signature of the form
\[ \sigma : s_1 \times \ldots \times s_k \to (p,0), \]
where $p \in \{ ty,tm \}$.
\end{lem}
\begin{proof}
Let $\mathbb{T}$ be an algebraic dependent type theory.
By \rlem{stable-con}, we may assume that $\mathbb{T}$ is contextual.
Then we define theory $\mathbb{T}'$ which has the same predicate symbols as $\mathbb{T}$.
For every $\sigma : (p_1,n_1) \times \ldots \times (p_k,n_k) \to (p,n)$ in $\mathcal{F}_0$, we add the following function symbol to $\mathbb{T}'$:
\[ \sigma' : (p_1,n_1) \times \ldots \times (p_k,n_k) \times (tm,0)^n \to (p,0). \]
Then we define $f(\sigma_0(\Gamma, x_1, \ldots x_k))$ as
\[ \sigma'_n(\Gamma, wk^{n,n_1}_{p_1,0}(\Gamma, x_1), \ldots wk^{n,n_k}_{p_k,0}(\Gamma, x_k), v_{n-1}, \ldots v_0). \]
For every predicate symbol $R$, we define $f(R(x_1, \ldots x_k))$ as $R(x_1, \ldots x_k)$.
For every axiom $\varphi \sststile{}{V} \psi$ of $\mathbb{T}$, we add axiom $f(\varphi) \sststile{}{V} f(\psi)$ to $\mathbb{T}'$.
Then $f$ is a morphism of theories $f : \mathbb{T} \to \mathbb{T}'$.

Moreover, there is a morphism $g : \mathbb{T}' \to \mathbb{T}$, which is defined as follows:
\begin{align*}
g(\sigma'_0(\Gamma, x_1, \ldots x_k, y_1, \ldots y_n)) & = subst^n_{p,0,0}(\Gamma, \sigma_0(\Gamma, x_1, \ldots x_k), y_1, \ldots y_n) \\
g(R(x_1, \ldots x_k)) & = R(x_1, \ldots x_k)
\end{align*}
The axioms of algebraic dependent type theories imply that $f$ and $g$ are inverses of each other.
\end{proof}

When we say that an algebraic dependent type theory is contextual (or presented in a contextual form), then we assume that it has a form as described in the previous lemma.

If an algebraic dependent type theory is presented in a contextual form, then every term is equivalent to a term in which substitution operations are applied only to variables.
We can as usual omit the first argument to $subst_{p,n,k}$.
Also, if $X : (p,n+k)$ and $a_1, \ldots a_k : (tm,n)$, then we write $X[a_1, \ldots a_k]$ for
\[ subst_{p,n,k}(X, v_{n-1}, \ldots v_0, a_1, \ldots a_k). \]

One last problem is that we often need to apply weakening operations to variables.
It is not convenient to do this explicitly, so we introduce named variables in our terms.
Let $Var$ be some fixed countable set of variables.
To distinguish these variable from the ones that we used before we will call the latter \emph{metavariables}.
First, we assume that every metavariable $X$ of sort $(p,n)$ is equipped with a sequence of variables of length $n$, which we call the context of this metavariable.
Usually, we do not specify the context of a metavariable explicitly since it can be inferred from formulas and terms in which this metavariable appears.

Second, every binding should be annotated with a variable.
In particular, instead of $A_1, \ldots A_n \vdash b$ we should write $x_1 : A_1, \ldots x_n : A_n \vdash b$
and instead of $\sigma(A^1_1, \ldots A^1_{n_1}.\ b_1, \ldots A^k, \ldots A^k_{n_k}.\ b_k)$ we should write
\[ \sigma((x_1 : A^1_1), \ldots (x_{n_1} : A^1_{n_1}).\ b_1, \ldots ((x_1 : A^k_1), \ldots (x_{n_k} : A^k_{n_k}).\ b_k) \]

Now, we may use variables instead of de Bruijn indices.
If a variable $x_i$ appears in a context $x_1, \ldots x_n$, then it is decoded into expression $v_{n-i}$.
Every metavariable should appear in a context where all variables from its context are available.
Then a metavariable $X$ with context $x_1, \ldots x_n$ should be replaced with expression $subst(X, x_1, \ldots x_n)$.
We may also write $X[x_{i_1} \mapsto a_{i_1}, \ldots x_{i_k} \mapsto a_{i_k}]$,
which is replaced with expression $subst(X, a_1, \ldots a_n)$, where $a_j = x_j$ if $j \notin \{ i_1, \ldots i_k \}$.
Finally, we may write $ft^i(X)$, which works like a metavariable with context $x_1, \ldots x_{n-i}$.

\section{Examples}

Now, let us describe a few examples of algebraic dependent type theories with substitution.
If we take their stabilization, then we get theories corresponding to usual constructions of the type theory.
Every theory is presented in the contextual form.
Also, to simplify the notation, we use the following agreement.
For every sequent of the form $\varphi \sststile{}{} \Gamma \vdash A\ type$, there is also sequent $\Gamma \vdash A\ type \sststile{}{} \varphi$
and, for every sequent of the form $\varphi \sststile{}{} \Gamma \vdash a : A$, there is also sequent $\Gamma \vdash a\!\downarrow\ \sststile{}{} \varphi$.

\begin{example}
The theory of unit types with eta rules has function symbols $\top : (ty,0)$ and $unit : (tm,0)$ and the following axioms:
\medskip
\begin{center}
\AxiomC{}
\UnaryInfC{$\vdash \top\ type$}
\DisplayProof
\quad
\AxiomC{}
\UnaryInfC{$\vdash unit : \top$}
\DisplayProof
\quad
\AxiomC{$\vdash t : \top$}
\UnaryInfC{$\vdash t \deq unit$}
\DisplayProof
\end{center}
\end{example}

\begin{example}
The theory of unit types without eta rules has function symbols $\top : (ty,0)$, $unit : (tm,0)$ and $\top\text{-}elim : (ty,1) \times (tm,0) \times (tm,0) \to (tm,0)$.
The axioms for $\top$ and $unit$ are the same, and the axioms for $\top\text{-}elim$ are
\medskip
\begin{center}
\AxiomC{$x : \top \vdash D\ type$}
\AxiomC{$\vdash d : D[x \mapsto unit]$}
\AxiomC{$\vdash t : \top$}
\TrinaryInfC{$\vdash \top\text{-}elim(x.\,D, d, t) : D[x \mapsto t]$}
\DisplayProof
\end{center}

\medskip
\begin{center}
\AxiomC{$x : \top \vdash D\ type$}
\AxiomC{$\vdash d : D[x \mapsto unit]$}
\BinaryInfC{$\vdash \top\text{-}elim(x.\,D, d, unit) \deq d$}
\DisplayProof
\end{center}
\end{example}

\begin{example}[sigma-eta]
The theory of $\Sigma$ types with eta rules has function symbols
\begin{align*}
\Sigma & : (ty,1) \to (ty,0) \\
pair & : (ty,1) \times (tm,0) \times (tm,0) \to (tm,0) \\
proj_1 & : (ty,1) \times (tm,0) \to (tm,0) \\
proj_2 & : (ty,1) \times (tm,0) \to (tm,0)
\end{align*}
and the following axioms:
\medskip
\begin{center}
\AxiomC{}
\UnaryInfC{$\vdash \Sigma(x.\,B)\ type$}
\DisplayProof
\quad
\AxiomC{$\vdash b : B[x \mapsto a]$}
\UnaryInfC{$\vdash pair(x.\,B, a, b) : \Sigma(x.\,B)$}
\DisplayProof
\end{center}

\medskip
\begin{center}
\AxiomC{$\vdash p : \Sigma(x.\,B)$}
\UnaryInfC{$\vdash proj_1(x.\,B, p) : ft(B)$}
\DisplayProof
\quad
\AxiomC{$\vdash p : \Sigma(x.\,B)$}
\UnaryInfC{$\vdash proj_2(x.\,B, p) : B[x \mapsto proj_1(x.\,B, p)]$}
\DisplayProof
\end{center}

\medskip
\begin{center}
\AxiomC{$\vdash b : B[x \mapsto a]$}
\UnaryInfC{$\vdash proj_1(x.\,B, pair(x.\,B, a, b)) \deq a$}
\DisplayProof
\qquad
\AxiomC{$\vdash b : B[x \mapsto a]$}
\UnaryInfC{$\vdash proj_2(x.\,B, pair(x.\,B, a, b)) \deq b$}
\DisplayProof
\end{center}

\medskip
\begin{center}
\AxiomC{$\vdash p : \Sigma(x.\,B)$}
\UnaryInfC{$\vdash pair(x.\,B, proj_1(x.\,B, p), proj_2(x.\,B, p)) \deq p$}
\DisplayProof
\end{center}
\end{example}

\begin{example}[sigma-no-eta]
The theory of $\Sigma$ types without eta rules has the following function symbols:
\begin{align*}
\Sigma & : (ty,1) \to (ty,0) \\
pair & : (ty,1) \times (tm,0) \times (tm,0) \to (tm,0) \\
\Sigma\text{-}elim & : (ty,1) \times (tm,2) \times (tm,0) \to (tm,0)
\end{align*}
The axioms for $\Sigma$ and $pair$ are the same, and the axioms for $\Sigma\text{-}elim$ are
\medskip
\begin{center}
\AxiomC{$z : \Sigma(x.\,B) \vdash D\ type$}
\AxiomC{$x : ft(B), y : B \vdash d : D'$}
\AxiomC{$\vdash p : \Sigma(x.\,B)$}
\TrinaryInfC{$\vdash \Sigma\text{-}elim(z.\,D, x y.\,d, p) : D[z \mapsto p]$}
\DisplayProof
\end{center}

\medskip
\begin{center}
\def\extraVskip{1pt}
\AxiomC{$z : \Sigma(x.\,B) \vdash D\ type$}
\AxiomC{$x : ft(B), y : B \vdash d : D'$}
\AxiomC{$\vdash b : B[x \mapsto a]$}
\TrinaryInfC{$\vdash \Sigma\text{-}elim(z.\,D, x y.\,d, pair(x.\,B, a, b)) \deq d[x \mapsto a, y \mapsto b]$}
\DisplayProof
\end{center}
\end{example}

where $D' = D[z \mapsto pair(x.\,B, x, y)]$.

\begin{example}[pi-eta]
The theory of $\Pi$ types with eta rules has function symbols
\begin{align*}
\Pi & : (ty,1) \to (ty,0) \\
\lambda & : (tm,1) \to (tm,0) \\
app & : (ty,1) \times (tm,0) \times (tm,0) \to (tm,0)
\end{align*}
and the following axioms:
\medskip
\begin{center}
\AxiomC{}
\UnaryInfC{$\vdash \Pi(x.\,B)\ type$}
\DisplayProof
\quad
\AxiomC{}
\UnaryInfC{$\vdash \lambda(x.\,b) : \Pi(x.\,ty(b))$}
\DisplayProof
\end{center}

\medskip
\begin{center}
\AxiomC{$\vdash f : \Pi(x.\,B)$}
\AxiomC{$\vdash a : ft(B)$}
\BinaryInfC{$\vdash app(x.\,B, f, a) : B[x \mapsto a]$}
\DisplayProof
\end{center}

\medskip
\begin{center}
\AxiomC{$\vdash a : ft(B)$}
\UnaryInfC{$\vdash app(x.\,B, \lambda(x.\,b), a) \deq b[x \mapsto a]$}
\DisplayProof
\quad
\AxiomC{$\vdash f : \Pi(x.\,B)$}
\UnaryInfC{$\vdash \lambda(y.\,app(x.\,B, f, y)) \deq b$}
\DisplayProof
\end{center}
\end{example}

\begin{example}[Id]
The theory of identity types has function symbols
\begin{align*}
Id & : (tm,0) \times (tm,0) \to (ty,0) \\
refl & : (tm,0) \to (tm,0) \\
J & : (ty,3) \times (tm,1) \times (tm,0) \times (tm,0) \times (tm,0) \to (tm,0)
\end{align*}
and the following inference rules:
\medskip
\begin{center}
\AxiomC{$\vdash ty(a) \deq ty(a')$}
\UnaryInfC{$\vdash Id(a, a')\ type$}
\DisplayProof
\quad
\AxiomC{}
\UnaryInfC{$\vdash refl(a) : Id(a, a)$}
\DisplayProof
\end{center}

\medskip
\begin{center}
\AxiomC{$x : A, y : A, z : Id(x, y) \vdash D\ type$}
\AxiomC{$x : A \vdash d : D'$}
\AxiomC{$\vdash p : Id(a, a')$}
\TrinaryInfC{$\vdash J(x y z.\,D, x.\,d, a, a', p) : D[x \mapsto a, y \mapsto a', z \mapsto p]$}
\DisplayProof
\end{center}

\medskip
\begin{center}
\AxiomC{$x : A, y : A, z : Id(x, y) \vdash D\ type$}
\AxiomC{$x : A \vdash d : D'$}
\BinaryInfC{$\vdash J(x y z.\,D, x.\,d, a, a, refl(a)) \deq d[x \mapsto a]$}
\DisplayProof
\end{center}
\medskip

where $A = ty(a)$ and $D' = D[y \mapsto x, z \mapsto refl(x)]$.
\end{example}

\begin{example}
We define an endofunctor $U$ on the category of algebraic dependent type theories.
For every such theory $\mathbb{T}$, theory $U(\mathbb{T})$ has the same symbols as $\mathbb{T}$,
but it also has a universe which is closed under all function symbols of $\mathbb{T}$.

Let $\mathbb{T}$ be an algebraic dependent type theory in a contextual form.
Then $U(\mathbb{T})$ has the same predicate symbols as $\mathbb{T}$ and the following function symbols:
\begin{align*}
U & : (ty,0) \\
El & : (tm,0) \to (ty,0) \\
\sigma & : s_1 \times \ldots \times s_k \to (p,0) \\
\sigma^U & : U(s_1) \times s_1 \times \ldots \times U(s_k) \times s_k \to (tm,0)
\end{align*}
for every function symbol $\sigma : s_1 \times \ldots \times s_k \to (p,0)$ of $\mathbb{T}$, where $U(p,n_i) = (tm,0) \times \ldots \times (tm,n_i)$.

Theory $U(\mathbb{T})$ has the following axioms:
\medskip
\begin{center}
\AxiomC{}
\UnaryInfC{$\vdash U\ type$}
\DisplayProof
\qquad
\AxiomC{$\vdash a : U$}
\doubleLine
\UnaryInfC{$\vdash El(a)\ type$}
\DisplayProof
\end{center}
\medskip

For every function symbol $\sigma : (p_1,n_1) \times \ldots \times (p_k,n_k) \to p$ of $\mathbb{T}$ and every $1 \leq i \leq k$, we add the following axioms to $U(\mathbb{T})$:
\[ \vdash \sigma^U(t_1, \ldots, t_m)\!\downarrow\ \sststile{}{t_1, \ldots t_m}\ \vdash \sigma^U(t_1, \ldots t_m) : U \]
\[ \sststile{}{V}\ \vdash El(\sigma^U(\ldots, a_1, \ldots a_{n_i+1}, b, \ldots)) \cong e_p(\sigma(\ldots, b|_{\varphi_i}, \ldots)), \]
where $a_1, \ldots a_{n_i+1}, b$ are variables that correspond to $i$-th variable in $\sigma$, $e_{ty}(x) = x$, and $e_{tm}(x) = ty(x)$, and $\varphi_i$ equals to
\[ \bigwedge_{1 \leq j \leq n_i+1} El(a_j) = ft^{n_i+1-j}(e_{p_i}(b)). \]

To define the rest of the axioms of $U(\mathbb{T})$, we need to introduce a few auxiliary functions.
For every set of variables $V$, we define a set $U(V)$ as follows:
\[ V \amalg \{ x^j : (tm,n-j)\ |\ x : (p,n) \in V, 0 \leq j \leq n \}. \]
Now, we define a function $U : Term_\mathbb{T}(V)_{(ty,n)} \to Term_{U(\mathbb{T})}(U(V))_{(tm,n)}$ as follows:
\begin{align*}
U(ft^j(e_p(x))) & = x^j \\
U(ft^{j+1}(e_p(\sigma_n(\Gamma, t_1, \ldots t_k)))) & = U(ft^j(\Gamma)) \\
U(e_p(\sigma_n(\Gamma, t_1, \ldots t_k))) & = \sigma^U_n(\Gamma, t_1', \ldots t_k'),
\end{align*}
where $t_i' = U(ft^{n_i}(e_{p_i}(t_i))), \ldots U(e_{p_i}(t_i)), t_i$.

We add all axioms of $\mathbb{T}$ to $U(\mathbb{T})$ and, for every axiom $\varphi \sststile{}{V} \psi$ of $\mathbb{T}$, we add the following axiom:
\[ U(\varphi) \land \bigwedge_{x \in V} \xi_x \sststile{}{U(V)} U(\psi), \]
where $U(R(t_1, \ldots t_k))$ equals to $R(t_1, \ldots t_k)$,
$U(t_1 = t_2)$ equals to $U(e_p(t_1)) = U(e_p(t_2)) \land t_1 = t_2$,
and $\xi_x$ equals to $(e_p(x) = El(x^0)) \land \bigwedge_{1 \leq j \leq n} ft(El(x^{j-1})) = El(x^j)$.

Finally, let us show that $U$ is a functor.
Let $f : \mathbb{T} \to \mathbb{T}'$ be a morphism of algebraic dependent type theories.
Then $U(f)$ is defined in the obvious way on $U$, $El$ and symbols from $\mathbb{T}$.
Define $U(f)(\sigma^U(\ldots, x^{n_i}_i, \ldots x^0_i, x_i \ldots))$ as
\[ U(e_p(f(\sigma(x_1, \ldots x_k))))|_{\bigwedge_{1 \leq i \leq k} \xi_{x_k}}. \]
It is easy to see that $U(f)$ is a morphism of contextual theories and that $U$ preserves identity morphisms and compositions.
Thus, $U$ is a functor.
\end{example}

\begin{example}
There is a natural map $\mathbb{T} \to U(\mathbb{T})$.
We define $U^\omega(\mathbb{T})$ as the colimit of the following sequence:
\[ \mathbb{T} \to U(\mathbb{T}) \to U^2(\mathbb{T}) \to \ldots \]
Then $U^\omega(\mathbb{T})$ is the theory with a hierarchy of universes closed under constructions of $\mathbb{T}$.
\end{example}

\bibliographystyle{amsplain}
\bibliography{ref}

\providecommand{\bysame}{\leavevmode\hbox to3em{\hrulefill}\thinspace}
\providecommand{\MR}{\relax\ifhmode\unskip\space\fi MR }
% \MRhref is called by the amsart/book/proc definition of \MR.
\providecommand{\MRhref}[2]{%
  \href{http://www.ams.org/mathscinet-getitem?mr=#1}{#2}
}
\providecommand{\href}[2]{#2}
\begin{thebibliography}{10}

\bibitem{LPC}
J.~Ad{\'a}mek and J.~Rosick{\'y}, \emph{Locally presentable and accessible
  categories}, Cambridge University Press, 1994.

\bibitem{GAT}
John Cartmell, \emph{Generalised algebraic theories and contextual categories},
  Annals of Pure and Applied Logic \textbf{32} (1986), 209 -- 243.

\bibitem{CoC}
Thierry Coquand and G{\'e}rard Huet, \emph{The calculus of constructions},
  Information and Computation \textbf{76} (1988), no.~2, 95 -- 120.

\bibitem{alg-models}
V.~{Isaev}, \emph{Model structures on categories of models of type theories},
  (2016), \href {http://arxiv.org/abs/1607.07407} {\path{arXiv:1607.07407}}.

\bibitem{elephant}
Peter~T. Johnstone, \emph{Sketches of an elephant : a topos theory compendium},
  Oxford Logic Guides, Clarendon Press, Oxford, 2002, Autre tirage : 2008.

\bibitem{luo94}
Zhaohui Luo, \emph{{C}omputation and {R}easoning: {A} {T}ype {T}heory for
  {C}omputer {S}cience}, {O}xford {U}niversity {P}ress, 1994.

\bibitem{manes-algebraic-theories}
Ernest~G. Manes, \emph{Algebraic theories}, Graduate texts in mathematics,
  Springer, New York, Heidelberg, 1976.

\bibitem{MLTT73}
Per Martin-L{\"o}f, \emph{An intuitionistic theory of types: predicative part},
  Logic {C}olloquium '73 ({B}ristol, 1973), North-Holland, Amsterdam, 1975,
  pp.~73--118. Studies in Logic and the Foundations of Mathematics, Vol. 80.
  \MR{0387009 (52 \#7856)}

\bibitem{MLTT79}
\bysame, \emph{Constructive mathematics and computer programming}, Logic,
  methodology and philosophy of science, {VI} ({H}annover, 1979), Stud. Logic
  Found. Math., vol. 104, North-Holland, Amsterdam, 1982, pp.~153--175.
  \MR{682410 (85d:03112)}

\bibitem{MLTT72}
\bysame, \emph{An intuitionistic theory of types}, Twenty-five years of
  constructive type theory ({V}enice, 1995), Oxford Logic Guides, vol.~36,
  Oxford Univ. Press, New York, 1998, pp.~127--172. \MR{1686864}

\bibitem{PHL}
E.~Palmgren and S.J. Vickers, \emph{Partial horn logic and cartesian
  categories}, Annals of Pure and Applied Logic \textbf{145} (2007), no.~3, 314
  -- 353.

\bibitem{pitts}
A.~M. Pitts, \emph{Categorical logic}, Handbook of Logic in Computer Science,
  Volume 5. Algebraic and Logical Structures (S.~Abramsky, D.~M. Gabbay, and
  T.~S.~E. Maibaum, eds.), Oxford University Press, 2000, pp.~39--128.

\bibitem{streicher}
Thomas Streicher, \emph{Semantics of type theory: Correctness, completeness,
  and independence results}, Birkhauser Boston Inc., Cambridge, MA, USA, 1991.

\bibitem{b-systems}
V.~{Voevodsky}, \emph{{B-systems}},  (2014), \href
  {http://arxiv.org/abs/1410.5389} {\path{arXiv:1410.5389}}.

\bibitem{c-systems}
\bysame, \emph{{Subsystems and regular quotients of C-systems}},  (2014), \href
  {http://arxiv.org/abs/1406.7413} {\path{arXiv:1406.7413}}.

\end{thebibliography}

\end{document}